\documentclass[10 pt, leqno]{article}
\date{}
\usepackage{amssymb, amsbsy, amsmath, amsfonts, amssymb, amscd, mathrsfs, amsthm}
\usepackage[english]{babel}
\usepackage[T1]{fontenc}
\usepackage{indentfirst}
\usepackage{color}
\usepackage{enumerate}
 
\makeatletter
\@addtoreset{equation}{section}
\makeatother

\newtheorem{statement}{}[section]
\newtheorem{theorem}[statement]{Theorem}
\newtheorem{lemma}[statement]{Lemma}

\newtheorem{proposition}[statement]{Proposition}
\newtheorem{definition}[statement]{Definition}

\newcommand\C{\mathbb C}
\newcommand\N{\mathbb N}
\newcommand\R{\mathbb R}
\newcommand\T{\mathbb T}
\newcommand\D{\mathbb D}

\newcommand\e{{\rm e}}

\newcommand\eps{\varepsilon}
\newcommand\ind{1 \kern - 0.28 em {\rm I}}
\newcommand\dis{\displaystyle}
\renewcommand \Re{{\mathfrak R}{\rm e}\,}

\newcommand\converge{\mathop{\longrightarrow}\limits}

\let\amphi=\phi
\let\phi=\varphi
\let\hat = \widehat
\let\tilde=\widetilde
\newcommand\tq{\, ; \ }

\title{\bf Compactification, and beyond, of composition operators on Hardy spaces by weights}
\author{\it Pascal~Lef\`evre, Daniel~Li,  \\ \it Herv\'e~Queff\'elec, Luis~Rodr{\'\i}guez-Piazza}

\date{\footnotesize \today}

\begin{document}

\maketitle

\noindent {\bf Abstract.} We study when multiplication by a weight can turn a non-compact composition operator on $H^2$ into a compact operator, and when 
it can be in Schatten classes. The $q$-summing case in $H^p$ is considered.  
We also study when this multiplication can turn a compact composition operator into a non-compact one. 
\medskip

\medskip


\section{Introduction}

Let $\phi \colon \D \to \D$ be an analytic self-map and $C_\phi \colon H^2 \to H^2$ be the associated composition operator $f \mapsto f \circ \phi$. 
For $w \in H^2$, the multiplication operator $M_w$ is defined formally by $f \mapsto w f$ and the weighted composition operator by 
$f \mapsto w\, (f \circ \phi)$. It is known (see \cite{GKP} for instance) that twisting $C_\phi$ by some $M_w$ can improve 
its compactness properties, and even its membership in Schatten classes $S_p$ or the decay of its approximation numbers (\cite[Theorem~2.3]{GDHL}). 

In this note, we study, in a rather qualitative way, the following problem: given a symbol $\phi$, when can we find a non-trivial $w \in H^2$ such that 
$M_w$ has a smoothing effect on $C_\phi$, namely when is $M_{w} C_\phi$ compact if $C_\phi$ was not? Or the other way round: when can we find $w$ 
such that $M_{w} C_\phi$ is not compact if $C_\phi$ was? 

In \cite[Proposition~2.4]{DHL}, it is proved that for $M_w C_\phi$ to be compact for some $w \in H^2$ ($w \not \equiv 0$), it is necessary that:
\begin{equation} \label{not exposed}
m (\{|\phi^\ast| = 1 \}) = 0 \, , 
\end{equation} 
where $m$ is the normalized Lebesgue measure on $\T$ and $\phi^\ast$ the boundary values function of $\phi$. On the other hand, in order that $M_w C_\phi$ 
be Hilbert-Schmidt for some $w \in H^2$, $w \not\equiv 0$, it is sufficient that:
\begin{equation} \label{not extreme}
\int_{\T} \log (1 - |\phi^\ast|) \, dm > - \infty \, 
\end{equation} 
(\cite[Proposition~2.5]{DHL}). Note that \eqref{not exposed} means that $\phi$ is not an exposed point of the unit ball of $H^\infty$ (\cite{Amar}), and that 
\eqref{not extreme} means that it is not an extreme point of this unit ball (\cite[Theorem~7.9]{Duren}). 
\smallskip

There is a gap between these two conditions. The purpose of this work to fill this gap in several respects, this filling explaining in passing the initial gap.
\smallskip

In Section~\ref{compactification}, we show that condition \eqref{not exposed} is necessary and sufficient to have a compact weighted composition operator. 
We also give examples showing how small approximation numbers we can obtain. In Section~\ref{HS et Schatten}, we show that condition 
\eqref{not extreme} is necessary and sufficient to get a Hilbert-Schmidt weighted composition operator, and we show that it is also necessary and sufficient 
for getting a weighted composition operator in some, or all, Schatten classes. In Section~\ref{Hp}, we consider the case of $H^p$ spaces and study the 
nuclearity and the summing properties of the weighted composition operators. In Section~\ref{decompactification} we show that a composition operator 
can become non-compact by weighting it if and only if the image of the symbol  touches the boundary of the unit disk.

\section{Notation} \label{notation}

Let $\D$ be the open unit disk. The Hardy space $H^p$, $1 \leq p < \infty$, is the space of analytic functions $f \colon \D \to \C$ such that:
\begin{displaymath} 
\| f \|_p^p := \sup_{0 < r < 1} \frac{1}{2 \pi} \int_0^{2 \pi} |f (r \e^{i t})|^p \, dt < \infty \, .
\end{displaymath} 
Such functions have non-tangential limits $f^\ast (\e^{i t})$ almost everywhere on $\T = \partial \D$ and we have:
\begin{displaymath} 
\| f \|_p^p = \frac{1}{2 \pi} \int_0^{2 \pi} |f^\ast (\e^{i t})|^p \, dt \, .
\end{displaymath} 
For $p = 2$, $H^2$ is equivalently the space of analytic functions in $\D$ that can be written $f (z) = \sum_{n = 0}^\infty c_n z^n$ with 
$\| f \|_2^2 = \sum_{n = 0}^\infty |c_n|^2 < \infty$. In the sequel, for convenience, we write simply $\| \, . \, \|_2 = \| \, . \, \|$.
\smallskip

Any analytic self-map $\phi \colon \D \to \D$ induces a bounded operator $C_\phi \colon H^p \to H^p$, called the composition operator of 
\emph{symbol} $\phi$.
\smallskip

For $w \in H^p$, the multiplication operator $M_w$ is defined, formally, by $M_w f = w \, f$, and the \emph{weighted composition operator} $M_w C_\phi$ 
by $(M_w C_\phi) (f) = w \, (f \circ \phi)$. Note that to get $M_w C_\phi \colon H^p \to H^p$, it is necessary to have $w \in H^p$ since 
$(M_w C_\phi) (\ind) = w$. \emph{Throughout this paper it will be assumed that $w \in H^p$, and that $w \not\equiv 0$}. This membership is not sufficient 
in general; however $w \in H^\infty$ is sufficient (but not necessary!), since $H^\infty$ is the set of multipliers of $H^p$. Note that we may consider the bounded 
operator $M_w C_\phi$, even if $M_w$ is not bounded. 
\smallskip\goodbreak

Except in Section~\ref{Hp}, we work only with the Hilbert space $H^2$. 
\smallskip

\smallskip\goodbreak

For convenience, we will adopt in this paper the following terminology.

\begin{definition}
We say that the symbol $\phi$ is:
\begin{itemize}
\setlength\itemsep{1 pt}
\item [-] \emph{compactifiable} if $M_w C_\phi$ is compact for some $w \in H^2$ with $w \not\equiv 0$;
\item [-] \emph{decompactifiable} if $M_w C_\phi$ is bounded but not compact for some $w \in H^2$.
\end{itemize}
\end{definition}

For $\xi \in \T = \partial{\D}$ and $0 < h < 1$, the Carleson window $W (\xi, h)$ is defined as:
\begin{equation} 
W (\xi, h) = \{z \in \D \tq 1 - h \leq |z| \text{ and } |\arg (z \overline{\xi})| \leq \pi h \} \, .
\end{equation} 
If $\mu$ is a positive measure on $\overline \D$, the Carleson function of $\mu$ is:
\begin{equation} 
\rho_\mu (h) = \sup_{\xi \in \T} \mu [ \overline{W (\xi, h)}] \, .
\end{equation} 
The measure $\mu$ is called a Carleson measure when $\rho_\mu (h) = {\rm O}\, (h)$, and a vanishing Carleson measure when $\rho_\mu (h) = {\rm o}\, (h)$. 
By the Carleson embedding theorem, this is equivalent to say that the canonical inclusion $J_\mu \colon H^2 \to L^2 (\mu)$ is respectively bounded or 
compact.
\smallskip

It is convenient to coin the Hastings-Luecking box ${\tilde W} (\xi, h) \subseteq W (\xi, h) $ defined by:
\begin{equation} 
{\tilde W} (\xi, h) = \{z\in \D \tq 1 -  h \leq |z| < 1 - h/2 \text{ and } - \pi h < \arg (z \overline{\xi}) \leq \pi h \} \, .  
\end{equation} 

We denote $m$ the Haar measure (normalized Lebesgue measure) of $\T$. For a symbol $\phi$, $m_\phi = \phi^{\ast} (m)$ is the pull-back measure of $m$ by 
$\phi^{\ast} \colon \T\to \C$,  the (almost everywhere defined) radial limit function associated with $\phi$:
\begin{equation} 
\phi^{\ast} (\xi) = \lim_{r \to 1^{-}} \phi (r \xi) \, .
\end{equation} 
By definition $m_\phi (B) = m[ {\phi^\ast}^{- 1} (B)]$ for all Borel sets $B \subseteq \overline{\D}$. This measure $m_\phi$ is always a Carleson measure, 
due to the Littlewood subordination principle.
\smallskip

The Carleson function of $\phi$ is that of $m_\phi$ and is denoted $\rho_\phi$:
\begin{equation} 
\rho_\phi (h) = \sup_{\xi \in \T} m \big( {\phi^\ast}^{- 1} [\overline{W (\xi, h)}] \big) \, .
\end{equation} 

When the composition operator $C_\phi$ is compact on $H^2$, we have $|\phi^\ast| < 1$ a.e., and $m_\phi$ is supported by $\D$. Moreover, 
$m_\phi$ is then a vanishing Carleson measure.
\smallskip

Recall that a compact operator $T$ between separable Hilbert spaces $H_1$ and $H_2$ is in the Schatten class $S_p = S_p (H_1, H_2)$, $ p > 0$, if 
$\sum_{n \geq 0} [s_n (T)]^p < \infty$, where $\big( s_n (T) \big)$  is the sequence of singular numbers of $T$, i.e. the eigenvalues, arranged in non-increasing 
order, of $|T| = \sqrt{T^\ast T}$. For $p = 2$, $S_2 (H_1, H_2)$ is the Hilbert-Schmidt class. Let us also recall that, for $p \geq 2$, we have $T \in S_p$ if and 
only if $\sum_n \| T e_n\|^p < \infty$ for every orthonormal basis $(e_n)$ of $H_1$, and, for $ p \leq 2$, we have $T \in S_p$ if and only if 
$\sum_n \| T e_n\|^p < \infty$ for some orthonormal basis $(e_n)$ of $H_1$ (see \cite{Hu-Khoi-Zhu} for instance). It follows that if 
$S, T \colon H_1 \to H_2$ are two compact operators such that $\| S x \| \leq \| T x \|$ for all $x \in H_1$, then, for all $p > 0$, $T \in S_p$ implies 
$S \in S_p$.
\smallskip

We recall Luecking's theorem (\cite{Luecking}).
\begin{theorem} [Luecking's theorem] \label{Theo Luecking} 
Let $\mu$ be a positive Borel measure on $\D$. Then the canonical inclusion $J_\mu \colon H^2 \to L^2 (\mu)$ is in the Schatten class $S_p$, $p > 0$, 
if and only if:
\begin{displaymath} 
\sum_{n = 0}^\infty \sum_{j = 0}^{2^n - 1} [ 2^n \mu ({\tilde W}_{n, j}) ]^{p / 2} < \infty \, ,
\end{displaymath} 
where ${\tilde W}_{n, j} = {\tilde W} (\e^{ 2 j i \pi / 2^n} , 2^{ - n})$. 
\end{theorem} 

Let us point out that the above condition can be replaced by the following variant (\cite[Proposition 3.3]{LLQR-2008}):
\begin{displaymath} 
\sum_{n = 0}^\infty \sum_{j = 0}^{2^n - 1} [ 2^n \mu (W_{n, j}) ]^{p / 2} < \infty \, ,
\end{displaymath} 
where $W_{n, j} = W (\e^{ 2 j i \pi / 2^n} , 2^{ - n})$.
\medskip

As usual, the notation $A \lesssim B$ means that $A \leq c \, B$ for some positive constant $c$, and $A \approx B$ means that $A \lesssim B$ and 
$B \lesssim A$.
\goodbreak

\section{Compactification} \label{compactification} 

 \begin{theorem}\label{lille1} 
An analytic self-map $\phi \colon \D \to \D$ is compactifiable if and only if $m (\{|\phi^\ast| = 1\}) = 0$.
\end{theorem}
\begin{proof}
The necessary part is proved in \cite[Proposition~2.4]{DHL}. Let us recall  the easy proof of this fact. 

Indeed, suppose that $T = M_w \, C_\varphi$ is compact and that $|\phi^\ast| = 1$ on $E$, with $m (E) > 0$. Since $(z^n)_n$ converges weakly to $0$ 
in $H^2$ and since $T (z^n) = w \, \varphi^{n}$, we should have: 
\begin{displaymath} 
\int_{E} |w^\ast|^2 \, dm = \int_{E} |w^\ast|^2 |\phi^\ast|^{2 n} \, dm 
\leq \int_{\T} |w^\ast|^2 |\phi^\ast|^{2 n} \, dm =  \Vert T (z^n) \Vert^2  \converge_{n \to \infty} 0 \, ;
\end{displaymath} 
but this would imply that $w$ is null a.e. on $E$ and hence $w \equiv 0$ (see \cite{Duren}, Theorem~2.2), which was excluded.
\smallskip

Let us now prove the sufficient condition.

Assume that $m (\{|\phi^\ast| = 1 \}) = 0$ holds. Given $w \in H^2$, we can write:
\begin{displaymath} 
\Vert M_{w} C_{\phi} (f) \Vert^2 = \int_{\T}| w^\ast |^2 |f \circ \phi^\ast|^2 \, dm = \int_{\D} |f|^2 \, d\nu \, ,
\end{displaymath} 
where $\nu  = \nu_w = \phi^{\ast} (|w^\ast|^2 m)$, that is $\nu (A) = \int_{{\phi^{\ast}}^{- 1} (A)}|w^\ast|^2 \, dm$. By the Carleson embedding 
theorem (see \cite[page 129]{COMA}), a necessary and sufficient condition for the operator $M_{w} C_\phi \colon H^2 \to H^2$ to be compact is that 
$\nu$ is a vanishing Carleson measure for $H^2$. We now produce a suitable $w$, $w \not\equiv 0$. 

Let:
\begin{equation} \label{corona}
\Gamma_h = \{z \tq 1 - h \leq |z| < 1 \} 
\end{equation} 
and set:
\begin{displaymath} 
F_n = {\phi^{\ast}}^{- 1} (\Gamma_{2^{- n}}) \qquad \text{and} \qquad c_n = m (F_n) \, .
\end{displaymath} 
Our assumption implies that $c_n \converge_{n \to \infty} 0$. We can hence find an increasing sequence $(k_n)_{n \geq 1}$ of integers such that:
\begin{equation} \label{sub} 
\sum_{n = 1}^\infty c_{k_n} \log n < \infty \, .
\end{equation}
Let $\amphi_n \colon \T \to \R^+$ be defined as:
\begin{displaymath} 
\amphi_n = 
\left\{
\begin{array} {rcl}
\displaystyle \frac{1}{n} & \text{on} & F_{k_n} \, , \smallskip \\ 
1 & \text{on} & \T \setminus F_{k_n} \, .
\end{array}
\right.
\end{displaymath} 
Let $w_n$ be the associated outer function, satisfying $|w_{n}^{\ast}|=\amphi_n$, namely $w_n = \exp \, (- \psi_n)$, with :
\begin{displaymath} 
\psi_{n}(z) = \int_{\T} \frac{1 + z \, \e^{- i t}}{1 - z \, \e^{- i t}} \log \frac{1}{\amphi_{n} (t)} \, dm (t)
=\log n \int_{F_{k_n}} \frac{1 + z \, \e^{- i t}}{1 - z \, \e^{-i t}} \, dm (t) \, .
\end{displaymath} 
Observe that $\Re \psi_{n} (z) = \log n \int_{F_{k_n}} P_{z} (t) \, dm (t)$, where $P_{z} (t) = \frac{1 - |z|^2}{|1 - z \, \e^{-i t}|^2}$ is the Poisson 
kernel, so that $\Re \psi_{n}(z) \geq 0$ and $|w_{n} (z)|\leq 1$. Moreover $|w_{n}^{\ast}| = \frac{1}{n}$ on $F_{k_n}$. 

The condition \eqref{sub} ensures that the infinite product $w  = \prod_{n} w_n$ converges uniformly on compact subsets of $\D$, and defines a function 
$w \in H^\infty$, bounded by $1$ and without zeros. Indeed, since $\Re \psi_n \geq 0$, we see that:
\begin{displaymath} 
|1 - w_{n} (z)| \leq |\psi_{n} (z)| \leq \log n \int_{F_{k_n}}\frac{1 + |z|}{1 - |z|} \, dm (t) = (c_{k_n} \log n) \, \frac{1 + |z|}{1 - |z|} \, ;
\end{displaymath} 
subsequently, the series $\sum (1 - w_{n})$ converges normally on compact subsets of $\D$, and the infinite product $\prod w_n$ converges uniformly on 
compact subsets of $\D$, as claimed. 

The weighted composition operator $M_{w} C_\phi$ is bounded since $w \in H^\infty$.

Let finally $0 < h < 2^{- k_1}$ and $n = n (h)$ such that $2^{- k_{n + 1}} \leq h < 2^{- k_n}$. Let $\xi \in \T$. Then $W (\xi, h) \subseteq \Gamma_h$, 
so that:
\begin{displaymath} 
{\phi^{\ast}}^{- 1}[W (\xi, h)] \subseteq {\phi^{\ast}}^{- 1} (\Gamma_h) \subseteq {\phi^{\ast}}^{- 1} (\Gamma_{2^{- k_n}}) = F_{k_n} \, .
\end{displaymath} 

As a consequence, $|w^\ast (u)| \leq |w_{n}^\ast (u)| \leq \frac{1}{n}$ for all $u \in {\phi^{\ast}}^{- 1} [W (\xi, h)]$, and:
\begin{displaymath} 
\nu [W (\xi, h)] = \int_{{\phi^\ast}^{- 1} [W (\xi, h)]} |w^\ast|^{2} \, dm \leq \frac{1}{n^2} \, m_{\phi} [W (\xi, h)] \leq  \frac{1}{n^2} \, C h \, ,
\end{displaymath} 
because we know (see \cite[page 129]{COMA}) that $m_\phi$ is a Carleson measure. This ends the proof, since $n = n (h)$ tends to $\infty$ when $h$ goes to $0$.
\end{proof}

\noindent {\bf Remark.} The previous argument can be sometimes quantified, and the degree of compactness of $M_w C_\varphi$ specified  (even if there are 
limitations, as shown by the forthcoming Theorem~\ref{lens}).  

\begin{theorem} \label{specif} 
For each $\gamma$ with $0 < \gamma < 1/2$, there  exist a non-compact composition operator $C_\phi \colon H^2 \to H^2$ and a weight 
$w \in H^\infty$ such that, for some constant $b > 0$, we have:
\begin{displaymath} 
a_n (M_w C_\phi) \lesssim \exp (- b \, n^\gamma) \, .
\end{displaymath} 

In particular $M_w C_\phi$ belongs to all Schatten classes $S_p (H^2)$, $p > 0$.
\end{theorem} 

For the proof, we recall the following simple result.

\begin{proposition} \label{simp} 
Let $\nu$ be a vanishing Carleson measure on $\D$. Then:
\begin{displaymath} 
a_{n} (J_\nu) \lesssim \inf_{0 < h < 1} \bigg( \e^{- n h} + \sup_{0 \leq t \leq h} \sqrt{\frac{\rho_\nu (t)}{t}} \, \bigg) \, \raise 1 pt \hbox{,}
\end{displaymath} 
where $J_\nu \colon H^2\to L^{2}(\nu)$ is the canonical inclusion.

In particular, if $w \in H^\infty$ and $\varphi$ is a symbol, we have:
\begin{displaymath} 
a_{n} (M_w C_\varphi) \lesssim \inf_{0 < h < 1} \bigg( \e^{- n h} + \sup_{0 \leq t \leq h} \sqrt{\frac{\rho_\nu (t)}{t}} \, \bigg) \, \raise 1 pt \hbox{,}
\end{displaymath} 
where $\nu = \phi^\ast (|w^\ast|^2 m)$ is the pull-back measure of $|w^\ast|^2 m$ by $\phi^\ast$.
\end{proposition}

For the proof of Proposition~\ref{simp}, we refer to \cite[Theorem 5.1]{DDHL},  where the result is given only for composition operators, but working exactly 
the same for inclusions, except only that we have to replace the quantity $\sqrt{\rho_\nu (h) /h}$ by $\sup_{0 \leq t \leq h} \sqrt{\rho_\nu (t) / t}$. For the 
special case, just use that $\| J_\nu f \| = \| (M_w C_\varphi) f \|$ for all $f \in H^2$, so there exist two contractions $U \colon L^2 (\nu) \to H^2$ and 
$V \colon H^2 \to H^2$ such that $(M_w C_\phi) = U J_\nu$ and $J_\nu = V (M_w C_\phi)$, and hence $a_n (M_w C_\phi) = a_n (J_\nu)$.

%
\begin{proof} [Proof of Theorem~\ref{specif}] 

We use a construction made in \cite[Section 3.2]{LLQR-2008}.

Let $1 < \beta \leq 2$ and:
\begin{equation} 
u (t) = | \sin (t /2) |^\beta \, . 
\end{equation} 
There is an analytic function $U \colon \D \to \Pi^+ = \{ \Re z > 0 \}$ whose boundary values are:
\begin{equation} 
U^\ast (\e^{it}) = u (t) + i \, {\mathcal H} u (t) \, ,
\end{equation} 
where ${\mathcal H}$ is the Hilbert transform. The symbol $\phi$ is defined, for $z \in \D$, as:
\begin{equation} 
\phi (z) = \exp \big( - U (z) \big) \, .
\end{equation} 

By \cite[Lemma~3.6 and Lemma~4.3]{LLQR-2008}, the composition operator $C_\phi \colon H^2 \to H^2$ is not compact. 

Moreover, since $|\phi^\ast (\e^{it}) | = \exp \big( - |\sin (t / 2) |^\beta \big)$, we have:
\begin{displaymath} 
|\phi^\ast (\e^{it}) | \geq 1 - h \quad \Longleftrightarrow \quad |t| \leq \Big( \log \frac{1}{1 - h} \Big)^{1 / \beta} \approx h^{1 / \beta} \, ;
\end{displaymath} 
so, if $\Gamma_h$ is the annulus $ \{z \tq 1 - h \leq |z| < 1 \}$, and we set:
\begin{displaymath} 
F_k = {\phi^\ast}^{ - 1} (\Gamma_{2^{- k}}) \, ,
\end{displaymath} 
we have:
\begin{displaymath} 
\quad c_k := m (F_k) \approx 2^{- k / \beta} \, .
\end{displaymath} 

Now, let $\delta_k = \exp (- 2^{k / \beta} / k^2)$. We slightly modify the example of Theorem~\ref{lille1} as follows:
\begin{displaymath} 
\amphi_k = 
\left\{
\begin{array} {rcl}
\displaystyle \delta_k & \text{on} & F_k \, , \smallskip \\ 
1 \, & \text{on} & \T \setminus F_k \, .
\end{array}
\right.
\end{displaymath} 
Then, the series $\sum_{k \geq 1} c_k \log (1/\delta_k)$ converges since $c_k \log (1 / \delta_k) \lesssim 1 / k^2$. 
As in the proof of Theorem~\ref{lille1}, we can define an outer function $w$ such that $|w^\ast| = \prod_{k \geq 1} \amphi_k$. The same computation 
gives us, for any Carleson window $W (\xi, t)$ and for $\nu = \phi^\ast (|w^\ast|^2 m)$:
\begin{displaymath} 
\qquad \qquad \qquad \qquad \quad \nu [W (\xi, t)] \lesssim \delta_j^{\, 2} \, t \, , \quad \text{for} \quad 2^{- j - 1} \leq t < 2^{- j} \, . 
\end{displaymath} 

Let $0 < h < 1$ arbitrary. 

There exists an integer $l \geq 0$ such that $2^{- l - 1} \leq h < 2^{- l}$. Then for $0 < t \leq h$, we have 
$2^{- j - 1} \leq t < 2^{- j}$ for some $j \geq l$; hence:
\begin{displaymath} 
\frac{\rho_\nu (t)}{t} \lesssim \delta_j^{\, 2} \leq \delta_l^{\, 2} \, .
\end{displaymath} 
Therefore Proposition~\ref{simp} gives:
\begin{displaymath} 
a_n ( M_w C_\phi) \lesssim \inf _{l \in \N} (\e^{- n 2^{- l}} + \delta_l ) 
\lesssim \inf _{l  \geq 0} \big( \exp( - n 2^{- l}) + \exp ( - 2^{l /\beta} / l^2) \big) \, .
\end{displaymath} 

The choice $l = \Big[ \frac{\beta}{( \beta + 1) \log 2}\log n \Big]$ gives, for some $b > 0$:
\begin{displaymath} 
a_n ( M_w C_\phi) \lesssim \exp \big( - b \, n^{1 / (\beta + 1)} / (\log n)^2 \big) \, .
\end{displaymath} 

Now, if $0 < \gamma < 1/2$, we take $\beta$ such that $1 < \beta < \frac{1}{\gamma} - 1$ and $\beta \leq 2$. We obtain, with another $b > 0$:
\begin{displaymath} 
a_n ( M_w C_\phi) \lesssim \exp (- b \, n^\gamma ) \, ,
\end{displaymath} 
as claimed.
\end{proof}
\smallskip

\noindent {\bf Remark~1.} For $\beta < 1$, since we have $m_\phi (\Gamma_h) \approx h^{1/\beta}$, the composition operator $C_\phi$ is already 
compact. When $\beta = 1$, we have $m_\phi (\Gamma_h) \approx h$, but it can be checked that nevertheless $C_\phi$ is compact and 
$\rho_\phi (h) = {\rm O}\, \big(h / \log (1/h) \big)$ (see \cite[Remark~3, page 3117]{LLQR-2008}). Without doing that, we can use 
\cite[Theorem~4.1]{LLQR-2008} (which is an improvement of  \cite[Theorem~4.1]{LLQR-Mem}): there exists a compact composition operator with symbol 
$\tilde \phi$ such that $|{\tilde \phi}^\ast| = |\phi^\ast|$; therefore $m_{\tilde \phi} (\Gamma_h) = m_\phi (\Gamma_h) \approx h$.

For $\beta = 1$, the above proof only gives:
\begin{displaymath} 
a_n ( M_w C_\phi) \lesssim \exp \big( - b \, n^{1 / 2} / (\log n)^2 \big) \, .
\end{displaymath} 
Though in this case $C_\phi$ was already compact, that nevertheless allows to improve the compactness. 

\smallskip

\noindent {\bf Remark~2.} The case $\beta = 2$ corresponds to the simple symbol $\phi (z) = \frac{1 + z}{2}\,$. Indeed, we only used in our construction the 
modulus of the symbol and for this $\phi$, we have $|\phi^\ast (\e^{it}) | = |\cos (t/2)| \approx 1 - t^2/8 \approx \exp \big(-  |\sin (t/2 \sqrt{2})|^2 \big)$. 

We get the following result.
\begin{theorem} \label{specif-bis}
Let $\phi (z) = \frac{1 + z}{2}$. For each decreasing sequence $(\eps_k)$ of positive numbers such that $(\delta_k) = (2^{k/2} \eps_{2^k})$ is decreasing, 
there exist a weight $w \in H^\infty$ and a positive constant $b$ such that: 
\begin{displaymath} 
a_n (M_w C_\phi) \lesssim \exp \big( - b \, n^{1/3} \eps_n \big) \, .
\end{displaymath} 
\end{theorem} 
\begin{proof}
We only have to modify the proof of Theorem~\ref{specif}: we replace $F_k$ by:
\begin{displaymath}
F_k = {\phi^\ast}^{- 1} (\Gamma_{4^{k/3}})
\end{displaymath}
so $c_k = m (F_k) \approx 2^{- k/3}$, and we replace $\delta_k = \exp (- 2^{k / \beta} / k^2) = \exp ( - 2^{k / 2} / k^2)$ by:
\begin{displaymath} 
\delta_k = \exp ( - 2^{k / 3} \eps_{2^k}) \, ,
\end{displaymath} 
where $(\eps_k)_k$ is a given decreasing sequence of positive integers such that $(\delta_k)$ is decreasing. Note that, since $(\delta_k)$ is decreasing, we have 
$\eps_{2^k} \lesssim 2^{- k/2}$, so $\sum_k \eps_{2^k}  < \infty$. We get: 
\begin{displaymath} 
a_n (M_w C_\phi) \lesssim \inf_{l \geq 0} \big( \e^{- n 4^{- l/3}} + \e^{- 2^{l / 3} \eps_{2^l}} \big) \, ,
\end{displaymath} 
and, with $l = \big[\log n / \log 2 \big]$, we get, since $\eps_n \leq \eps_{2^l}$, for some $b > 0$:
\begin{displaymath} 
a_n (M_w C_\phi) \lesssim \exp \big( - b \, n^{1/3} \eps_n \big) \, . \qedhere
\end{displaymath} 
\end{proof}
\medskip

For example, with $\eps_k = 1/ (\log k)^2$, we get $a_n (M_w C_\phi) \lesssim \e^{( - b \, n^{1/3} / (\log n)^2 )} $.
\medskip

Theorem~\ref{specif-bis} improves a result of \cite[Theorem 2.3]{GDHL}, where for this symbol and a given $\alpha > 0$, weights $w$ are obtained such that:
\begin{displaymath} 
a_{n} (M_w C_\phi) \lesssim \bigg(\frac{\log n}{n}\bigg)^\alpha \, . 
\end{displaymath} 
%

\section{Hilbert-Schmidt and Schatten regularizations} \label{HS et Schatten}

We begin with a characterization of the symbols that can give a Hilbert-Schmidt weighted composition operator. 

\begin{theorem}\label{lens} 
An analytic self-map $\phi \colon \D \to \D$ can induce a Hilbert-Schmidt weighted composition operator $M_w C_\phi$, for some weight $w \in H^2$, if 
and only if:
\begin{displaymath} 
\int_{\T} \log \Big(\frac{1}{1 - |\phi^\ast|}\Big) \, dm<+ \infty \, .
\end{displaymath} 
\end{theorem}
\begin{proof}
That the condition is sufficient is proved in \cite[Proposition~2.5]{DHL}. For sake of completeness, we recall the argument. 

The hypothesis implies that there exists an outer function $w$ on $\D$ such that $|w^\ast|^2 = 1  - |\phi^\ast|$. Then, writing $T = M_w C_\phi$, we have:
\begin{displaymath} 
\sum_{n = 0}^\infty \| T (z^n) \|^2 = \sum_{n = 0}^\infty \int_\T (1 - |\phi^\ast|) |\phi^\ast|^{2 n} \, dm 
= \int_\T \frac{1}{1 + |\phi^\ast|} \, dm < + \infty \, ,
\end{displaymath} 
and $T$ is Hilbert-Schmidt, as claimed.
\smallskip 

Let us prove the necessity of the condition. 

If $w \in H^2$ exists such that $M_w C_\phi \colon H^2 \to H^2$ is Hilbert-Schmidt, we have in particular $|\phi^\ast| < 1$ $m$-almost everywhere, by 
the easy part of Theorem~\ref{lille1}. Since $M_w C_\phi$ is Hilbert-Schmidt, we have:
\begin{displaymath} 
\sum_{n = 0}^\infty \| w \, \phi^n \|^2 = \sum_{n = 0}^\infty \| (M_w C_\phi) (z^n) \|^2 < \infty \, ,
\end{displaymath} 
i.e.:
\begin{displaymath} 
\int_{\T} |w^\ast|^2 \, \frac{1}{1 - |\phi^\ast|^2} \, dm < \infty \, . 
\end{displaymath} 

The following lemma, with $u = |w^\ast|^2$, $v = 1 - |\phi^\ast|^2$ and $\alpha = 1$, then shows that 
$\int_\T \log \frac{1}{1 - |\phi^\ast|^2} \, d m < \infty$. In fact, since $w \in H^2$ and $w \not\equiv 0$, Jensen's inequality tells that the first condition of 
that lemma is satisfied. 
\end{proof}
\begin{lemma} \label{lemme utile}
Let $(\Omega, \nu)$ be a measure space and $u, v \colon \Omega \to (0, 1]$ measurable functions such that, for some $\alpha > 0$:
\begin{displaymath} 
\int_\Omega |\log u| \, d\nu < \infty  \qquad  \text{and} \qquad \int_\Omega u v^{- \alpha} \, d\nu < \infty \, .
\end{displaymath} 
Then $\displaystyle \int_\Omega |\log v |\, d\nu < \infty$.
\end{lemma}
\begin{proof}
If we set $g = v^{- \alpha}$ and $f = u v^{- \alpha}$, we have:
\begin{displaymath} 
0 \leq \log g = \log f + \log \frac{1}{u} \leq \log^{+} f + | \log u | \leq f  + | \log u | \, .
\end{displaymath} 
By hypothesis, $f$ (which is positive) and $|\log u|$ are integrable; hence $\log g$ is integrable and:
\begin{displaymath} 
\int_{\T} |\log v | \, d \nu < \infty \, . \qedhere
\end{displaymath} 
\end{proof}

In Theorem~\ref{lens}, we showed that for the outer function $w$ such that $|w^\ast|^2 = 1 - |\phi^\ast|$, the weighted composition operator $M_w C_\phi$ 
is Hilbert-Schmidt. For this weight, we cannot expect better in general, as said by the following theorem.

\begin{theorem} \label{HS pas S_p}
There exist a symbol $\phi$ satisfying $\int_\T \log (1 - |\phi^\ast|) \, dm > - \infty$ such that, if  $w$ is any outer function satisfying 
$|w^\ast| = 1 - |\phi^\ast|$, the weighted composition operator $M_w C_\phi$ is Hilbert-Schmidt, but $M_w C_\phi \notin S_p$, for all $p < 2$.
\end{theorem} 
\begin{proof}
Let, for $|t| \leq \pi$:
\begin{displaymath} 
u (t) = 1 - \exp ( - \, \e^{1 / |t|}) 
\end{displaymath} 
We have $0 < 1 - \exp (- \, \e^{1/\pi}) \leq u (t) \leq 1$; hence $\int_{- \pi}^\pi \log u (t) \, dt > - \infty$; therefore there is an outer function 
$\phi \in H^\infty$ such that $|\phi^\ast (\e^{it})| = u (t)$. 

Moreover, we also have $\int_\T \log (1 - |\phi^\ast|) \, dm = \int_{- \pi}^\pi \log \big(1 - u (t) \big) \, dt > - \infty$. Hence if $w$ is an outer function such 
that $|w^\ast|^2 = 1 - |\phi^\ast|$, the weighted composition operator $M_w C_\phi$ is Hilbert-Schmidt. We are going to show that $M_w C_\phi$ 
does not belong to any Schatten class for $p < 2$. 

For that, we use Theorem~\ref{Theo Luecking}. The weighted composition operator $M_w C_\phi$ can be viewed as an inclusion  
$J_\nu \colon H^2 \to L^2 (\nu)$, where $\nu = \phi^\ast (|w^\ast|^2 m)$. Here, we also have $d\nu (z) = (1 - |z|) \, d m_\phi (z)$. 

Since $p < 2$, we have:
\begin{displaymath} 
\sum_{j = 0}^{2^n - 1} [2^n \nu ({\tilde W}_{n, j}) ]^{p / 2} \geq \bigg(  \sum_{j = 0}^{2^n - 1} 2^n \nu ({\tilde W}_{n, j}) \bigg)^{p / 2} 
= [2^n \nu ({\tilde \Gamma}_{2^{ n}}) ]^{p / 2} \, ,
\end{displaymath} 
where ${\tilde \Gamma}_h = \{ z \in \D \tq 1 - h \leq |z| \leq 1 - h / 2 \}$. 

But $\nu ({\tilde \Gamma}_{2^{- n}})\approx  2^{- n} \, m_\phi ({\tilde \Gamma}_{2^{- n}})$ and 
\begin{displaymath} 
m_\phi ({\tilde \Gamma}_h) \approx \frac{1}{(\log 1 / h) (\log \log 1 / h)^2} \, \cdot
\end{displaymath} 
In fact, we have $\phi^\ast (\e^{it}) \in {\tilde \Gamma}_h$ if and only if $h / 2 \leq \exp ( - \, \e^{1 / |t|}) \leq h$, which is equivalent to:
\begin{displaymath} 
\frac{1}{\log \log 2 / h} \leq |t| \leq \frac{1}{\log \log 1 / h}
\end{displaymath} 
and:
\begin{align*} 
\frac{1}{\log \log 1 / h} - \frac{1}{\log \log 2 / h} 
& \approx \frac{1}{(\log \log 1 / h)^2}\, \log \bigg( 1 + \frac{\log 2}{\log 1 / h} \bigg) \\
& \approx \frac{1}{(\log 1 / h) (\log \log 1 / h)^2} \, \cdot 
\end{align*} 
Hence: 
\begin{displaymath} 
2^n \nu ({\tilde \Gamma}_{2^{ n}})  \gtrsim \frac{1}{n \, (\log n)^2} 
\end{displaymath} 
and we obtain:
\begin{displaymath} 
\sum_{n = 0}^{+\infty} \sum_{j = 0}^{2^n - 1} [2^n \nu ({\tilde W}_{n, j}) ]^{p / 2} 
\gtrsim \sum_{n = 0}^{+\infty} \frac{1}{n^{p / 2} \, (\log n)^p} = \infty \, ,
\end{displaymath} 
since $p / 2 < 1$. Luecking's theorem tells that $M_w C_\phi \notin S_p$.
\end{proof}
\medskip
If Theorem~\ref{HS pas S_p} does not allow to have better than Hilbert-Schmidt with the same weight, an improvement is possible by taking another weight.
\goodbreak

\begin{theorem} \label{Schatten p < 2}
Assume that the composition operator $C_\phi$ can induce a Hilbert-Schmidt weighted composition operator. Then there exists another weight $w \in H^2$ 
such that $M_w C_\phi \in S_p$ for every $p < 2$.
\end{theorem} 
\begin{proof}
By Theorem~\ref{lens}, we have $\int_\D\log \frac{1}{1 - |z|} \, dm_\phi(z) < \infty$.  
Take an integer $K > 1 / p$ and let $w_K$ be an outer function 
such that $|w_K^\ast| = (1 - |\phi^\ast|)^K$. 

We point out that 

$$\|w_K^\ast (\phi^\ast)^n\|_{L^\infty(\T)}\le\sup_{t\in(0,1)}(1-t)^Kt^{n}\lesssim \frac{1}{n^K}$$

Hence we have, for some positive constant $C$ (depending on $K$ but not on $n$):
\begin{displaymath} 
\| (M_{w_K} C_\phi) (z^n) \|^2 = \int_\T |w_K^\ast|^2 |\phi^\ast|^{2 n} \, dm \leq \frac{C}{n^{2K}} \, \cdot
\end{displaymath} 

It follows that $\| (M_{w_K} C_\phi) (z^n) \|^p \leq C^{p/2} / n^{Kp}$ and hence 
\begin{displaymath} 
\sum_{n = 1}^\infty \| (M_{w_K} C_\phi) (z^n) \|^p < \infty \, , 
\end{displaymath} 
since $K p> 1$. 

Now, by the du Bois-Reymond lemma, there exists a measurable function $g \colon [0, 1] \to \R_+$ such that $g (t) \converge_{t \to 1} \infty$ and 
$\int_\D g (|z|) \, \log \frac{1}{1 - |z|} \, dm_\phi (z) < \infty$. So there is an outer function $w$ such that 
$|w^\ast| = (1 - |\phi^\ast|)^{g \circ |\phi^\ast|}$. Since $g (t) \converge_{t \to 1} \infty$, we have $g (t) \geq K$ for $t$ close enough to $1$ and it follows 
that $|w^\ast| \lesssim |w_K^\ast|$ (up to a constant depending on $K$ only). Hence $\| (M_w C_\phi) f \| \lesssim \| (M_{w_K} C_\phi) f \|$ for all 
$f \in H^2$, and $M_w C_\phi \in S_p$ since $M_{w_K} C_\phi \in S_p$. 
\end{proof}
\goodbreak

\begin{theorem} \label{Schatten p > 2}
For every $p < \infty$, if $M_w C_\phi \in S_p$ for some weight $w$, then there exists another weight $\tilde w$ for which $M_{\tilde w} C_\phi$ is 
Hilbert-Schmidt. 
\end{theorem} 
\begin{proof}
For $p \leq 2$, this is obvious, with the same weight, since $S_p \subseteq S_2$. So we assume $p > 2$. We have 
$\sum_{n = 0}^\infty \| (M_w C_\phi) (z^n) \|^p < \infty$, i.e.:
\begin{displaymath} 
\sum_{n = 0}^\infty \bigg( \int_\T |w^\ast|^2 |\phi^\ast|^{2 n} \, dm \bigg)^p < \infty \, .
\end{displaymath} 
When $\sum_{n = 0}^\infty |c_n|^p < \infty$, the H\"older inequality implies that, for $\beta > 1/ q$ ($q$ is the conjugate exponent of $p$), we have:
\begin{displaymath} 
\sum_{n = 0}^\infty \frac{1}{n^\beta} \, |c_n| \leq \bigg( \sum_{n = 0}^\infty \frac{1}{n^{\beta q}} \bigg)^{1 / q} 
\bigg( \sum_{n = 0}^\infty |c_n|^p \bigg)^{1 / p} < \infty \, .
\end{displaymath} 
Now, 
\begin{displaymath} 
(1 - |\phi^\ast|^2)^{- \beta} = \sum_{n = 0}^\infty \binom{- \beta}{\, n} (- 1)^n |\phi^\ast|^{2 n} \, ,
\end{displaymath} 
and, by the Stirling formula $\binom{- \beta}{\, n} (- 1)^n \approx n^{\beta - 1}$. Hence if we take $\beta$ such that $1/ q < \beta < 1$ and set 
$\alpha = 1 - \beta$, we have $\alpha > 0$ and:
\begin{displaymath} 
\int_\T |w^\ast|^2 (1 - |\phi^\ast|^2)^{- \alpha} \, dm \approx \sum_{n = 0}^\infty \frac{1}{n^\beta} \int_\T |w^\ast|^2 |\phi^\ast|^{2 n} \, dm 
< \infty \, .
\end{displaymath} 
It follows from Lemma~\ref{lemme utile} that $\int_\T | \log (1 - |\phi^\ast|^2) | \, dm < \infty$, and then, from Theorem~\ref{lens}, that there is a 
weight $\tilde w$ for which $M_{\tilde w} C_\phi$ is Hilbert Schmidt. 
\end{proof}

Let us put together Theorem~\ref{lens}, Theorem~\ref{Schatten p < 2} and Theorem~\ref{Schatten p > 2}.

\begin{theorem} \label{theo recapitul} 
For any symbol $\phi$, the following assertions are equivalent:
\begin{itemize}
\item [$1)$] there is a weight $w$, with $w \in H^2$, such that $M_{w} C_\phi$ is Hilbert-Schmidt;

\item [$2)$] there is a weight $\tilde w$, with $\tilde w \in H^\infty$, such that $M_{\tilde w} C_\phi \in S_p$ for all $p > 0$;

\item [$3)$] there exist $p < \infty$ and a weight $w_p$, with $w_p \in H^\infty$, such that  $M_{w_p} C_\phi \in S_p$;

\item [$4)$] $\displaystyle \int_\T \log \frac{1}{1 - |\phi^\ast|} \, dm < \infty$.
\end{itemize}
\end{theorem} 
\smallskip

As a consequence, we see that in general, the condition $m (\{ |\phi| = 1\}) = 0$ cannot give better than a compactification. 
\begin{theorem} \label{pas Schatten}
There exists a compactifiable symbol $\phi$, i.e. $m (\{ |\phi^\ast| = 1\}) = 0$, such that, whatever the weight $w$, $M_w C_\phi$ is not in any Schatten class 
$S_p$, with $p < \infty$. 
\end{theorem} 
\begin{proof}
It suffices to find a symbol $\phi$ such that $m (\{ |\phi^\ast| = 1\}) = 0$ but such that $\int_\T \log \frac{1}{1 - |\phi^\ast|} \, dm = \infty$, i.e. an 
element of the unit ball of $H^\infty$ that is an extreme point of that unit ball but not en exposed point. If we set $u (t) = 1 - \e^{- 1 / |t|}$ for $|t| \leq \pi$, 
then $0 < 1 - \e^{- 1/\pi} \leq u (t) \leq 1$, so $\int_{|t|\leq \pi} \log u (t) \, dt > - \infty$, so there exists an outer function $\phi \in H^\infty$ such that 
$|\phi^\ast (\e^{it}) | = u (t)$. Clearly, this function works. 
\end{proof}
%

\section {Weighted composition operators on $H^p$} \label{Hp}

In this section we assume that $1 \leq p < +\infty$. We are interested here in finding a characterization of the symbols that can give a weighted composition 
operator belonging to some specific ideal of operators. In particular, we focus on the ideal of nuclear operators and the ideal of absolutely summing operators.

First let us recall 

\begin{itemize}
\item  [-] An operator $T \colon X \to Y$ between Banach spaces $X$ and $Y$ is nuclear if there are elements $y_n \in Y$ and linear forms 
$x^\ast_n \in X^\ast$ with $\sum_{n = 0}^\infty \|x^\ast_n\| \, \| y_n\| < \infty$ such that $T x = \sum_{n = 0}^\infty x^\ast_n (x) y_n$ for all 
$x \in X$.

\item [-] An operator $T \colon X \to Y$ between Banach spaces $X$ and $Y$ is $r$-summing, $1 \leq r < \infty$, if there is a positive constant $C$ such that:
\begin{displaymath} 
\bigg( \sum_{k = 1}^n \| T x_k \|^r \bigg)^{1 / r} 
\leq C \, \sup_{x^\ast \in B_{X^\ast}} \bigg( \sum_{k = 1}^n | \langle x^\ast , x_k \rangle |^r \bigg)^{1 / r}
\end{displaymath} 
for all finite sequence $(x_1, \ldots, x_n)$ in $X$.
\end{itemize}

The main result of this section is

\begin{theorem} 
Let $\phi \colon \D \to \D$ be a symbol.

The following assertions are equivalent.\par
\begin{enumerate}[{\rm (1)}]

\item There exists a weight $w$ such that $M_w C_\phi \colon H^p \to H^p$ is a nuclear operator for every $p\ge1$.

\item There exists a weight $w$ such that $M_w C_\phi \colon H^p \to H^p$ is $1$-summing for every $p\ge1$ (and hence is $r$-summing for every $r \geq 1$).

\item There exists a weight $w$ such that $M_w C_\phi \colon H^p \to H^p$ is $r$-summing for some $r \geq 1$ and some $p\ge1$.

\item $\displaystyle \int_\T \log \frac{1}{1 - |\phi^\ast|} \, dm < \infty$. 

\end{enumerate}
\end{theorem} 
\begin{proof}
Clearly (1) implies (2), which implies (3).
\smallskip

The weighted composition operator $(M_w C_\phi)$ can be viewed as the Carleson embedding $J_{\nu_p} \colon H^p \to L^p (\nu_p)$ where 
$\nu_p = \phi^\ast (|w^\ast|^p m)$ is a finite measure on $\D$. 
\smallskip

Assume (3). Then $J_{\nu_p}$ is actually $r$-summing on $H^s$ where $s=min(2,p)$ thanks to \cite[Theorem 8.4]{Lef-Rod}. By 
\cite[Proposition~2.3, 1)]{Lef-Rod}, we have:
\begin{displaymath} 
\int_\T \frac{|w^\ast|^p}{(1 - |\phi^\ast|)^{s / 2}} \, dm = \int_\D \frac{d \nu_p (z)}{(1 - |z|)^{s / 2}} < \infty \, .
\end{displaymath} 

By Lemma~\ref{lemme utile}, that implies that $\int_\T \log \frac{1}{1 - |\phi^\ast|} \, dm < \infty$ and (4) is satisfied.
\smallskip

Now assume that (4) is satisfied. For every $f \in H^p$, we denote by $\hat f (n)$ its $n^{th}$ Taylor coefficient. We point out that  the functional 
$f\in H^p\mapsto \hat f (n)$ has norm $1$. Then, for any operator $T \colon H^p \to Y$ satisfying  $\sum_{n = 0}^\infty \| T e_n \| < \infty$ where  
$e_n (z) = z^n$, it is easy to check that $T$ is a nuclear operator.

Our assumption implies that there exists an outer function $w$ such that $|w^\ast| = (1 - |\phi^\ast|)^2$ a.e. and we already pointed out that 
$\|w^\ast (\phi^\ast)^n\|_{L^\infty(\T)}\le \frac{C}{n^2}$, for some constant $C > 0$.

Hence:
\begin{displaymath} 
\| (M_{w} C_\phi) (e_n) \|_p = \bigg(\int_\T |w^\ast|^p |\phi^\ast|^{p n} \, dm \bigg)^\frac{1}{p} \leq \frac{C}{n^{2}} \,\cdot 
\end{displaymath} 

We get that $\dis \sum_n\| (M_{w} C_\phi) (e_n) \|_p<+\infty$ and hence that $(M_{w} C_\phi)$ is a nuclear operator.
\end{proof}
\goodbreak
\section{Decompactification}  \label{decompactification} 

\subsection{An initial example}

We refer to \cite[page~27]{SHAPI} (see also \cite{LELIQURO}) for the definition of the lens map $\lambda_\theta$ of parameter $\theta$, $0 < \theta < 1$.

We saw in \cite[Theorem~4.1]{GDHL} that multiplication by a second symbol $w$ can improve the degree of compactness of a composition operator 
$C_\varphi$. For example, if $\phi = \lambda_\theta$, which satisfies (\cite[Theorem~2.1]{LELIQURO}):
\begin{displaymath} 
\e^{- b_1 \sqrt{n}} \lesssim a_{n} (C_{\lambda_\theta}) \lesssim \e^{- b_2 \sqrt{n}} 
\end{displaymath} 
(implying in particular that $C_{\lambda_\theta}$ is in all Schatten classes $S_p (H^2)$, $p > 0$), we exhibited functions $w\in H^\infty$ such that:
\begin{displaymath} 
\e^{- b'_1 n/\log n} \lesssim  a_{n}(M_w C_\varphi) \lesssim \e^{- b'_2 n/\log n}.
\end{displaymath} 
We wish to prove here  that, conversely, multiplication by $w$ can in some sense ``decompactify'' $C_\phi$ while keeping it bounded. 
We shall begin with an explicit example.

\begin{theorem} \label{app} 
Let $\lambda_\theta$ be a lens map, $0 < \theta < 1$, and let $w (z) = (1 - \lambda_\theta (z))^{a}$ where 
$a = \frac{1}{2} \big(1 - \frac{1}{\theta} \big) < 0$. Then $w \in H^2$ and the weighted composition operator $M_w C_{\lambda_\theta}$ is 
bounded but not compact on $H^2$, though $C_{\lambda_\theta}$ is in all Schatten classes $S_p (H^2)$, $p > 0$. 
\end{theorem}
\begin{proof} 
We first observe that $w \in H^2$ since $|1 - \lambda_{\theta}^\ast (\xi)|\approx |1 - \xi|^\theta$ when $\xi \in \T$ (see \cite[Lemma~2.5]{LELIQURO}) 
and $2 \, a \, \theta = \theta - 1 > - 1$. Let now $f \in H^2$. Then we have, formally: 
\begin{displaymath} 
\Vert M_w C_\phi (f) \Vert^2 = \int_{\T} |1 - \lambda_{\theta}^\ast (\xi)|^{2 a} \, |f \circ \lambda_{\theta}^\ast (\xi)|^2 \, dm (\xi) 
= \int_{\D} |f (u)|^2 \, d\mu (u) \, ,
\end{displaymath} 
where:
\begin{displaymath} 
d \mu = |1 - u|^{2 a} \, d m_{\lambda_{\theta}} (u) \, ,
\end{displaymath}
with $m_{\lambda_{\theta}}= \lambda_{\theta}^{\ast} (m)$.

It is sufficient to prove that $\mu$ is a Carleson measure, but not a vanishing one, for $H^2$. We can restrict ourselves to the Carleson windows 
$W (1, h)$ centered at~$1$. 

We know (\cite[Lemma~2.5]{LELIQURO}) that,  for some constants $C > c > 0$, depending on $\theta$, we have 
$c \, |t|^\theta \leq 1 - |\lambda_\theta^\ast (\e^{it}) | \leq C \, |t|^\theta$ and 
$| \arg [\lambda_\theta^\ast (\e^{it}) ] | \leq C \pi \, |t|^\theta$; it follows easily that 
$m_{\lambda_{\theta}} [W (1, h)] \approx h^{1/\theta}$. 
Hence:
\begin{align*} 
\mu [W (1, h)] & = \sum_{n = 0}^\infty \mu [W (1, 2^{- n} h ) \setminus W (1, 2^{- n - 1} h)] \\ 
& \approx \sum_{n = 0}^\infty (2^{- n} h)^{2 a} m_{\lambda_{\theta}} [W (1, 2^{- n} h) \setminus W (1, 2^{- n - 1} h)]  \\
& \lesssim \sum_{n = 0}^\infty (2^{- n} h)^{2 a} (2^{- n} h)^{1/\theta} \lesssim h \sum_{n = 0}^\infty 2^{- n} = 2 h 
\end{align*} 
(since $2 a + 1 / \theta = 1$), proving that $\mu$ is a Carleson measure. 

On the other hand, if we consider the modified Hastings-Luecking windows:
\begin{displaymath} 
\tilde {\tilde W} (1, h) = \{ z \in \D \tq (c / 2 C) \, h \leq1- |z| \leq h \quad \text{and} \quad |\arg (z) | \leq \pi h \} \, ,
\end{displaymath} 
we have $m_{\lambda_{\theta}} \big(\tilde{\tilde W} (1, h) \big) \gtrsim h^{1/\theta}$, because if $(h / 2 C)^{1 / \theta} \leq |t| \leq (h / C)^{1 / \theta}$, 
we have $1 - |\lambda_\theta^\ast (\e^{it}) | \leq C \,|t|^\theta \leq h$, 
$ 1 - |\lambda_\theta^\ast (\e^{it}) | \geq c \, |t|^\theta \geq (c / 2 C) \, h$ 
 and $| \arg [\lambda_\theta^\ast (\e^{it}) ] | \leq C \pi \, |t|^\theta \leq \pi \, h$, so $\lambda_\theta^\ast (\e^{it}) \in \tilde {\tilde W} (1, h)$. 
It follows that:
\begin{align*}
\mu [W (1, h)] \geq \mu \big( \tilde{\tilde W} (1, h) \big) \gtrsim h^{2 a} m_{\lambda_{\theta}} \big(\tilde{\tilde W} (1, h) \big) 
\gtrsim h^{2 a} h^{1 / \theta} = h \, ,
\end{align*}
so $\mu$ is not a vanishing Carleson measure.
\end{proof}
 %

\subsection{The general case}

We now turn to the general case, with a less explicit construction, under the following form.
\goodbreak

\begin{theorem} \label{geca}  
An analytic self-map $\phi \colon \D\to \D$ is decompactifiable if and only if $\Vert \phi \Vert_\infty = 1$.  
\end{theorem}
\begin{proof}
First assume that $\Vert \phi \Vert_\infty < 1$. Let $w \in H^2$ and $(f_n)$ a weakly null sequence in $H^2$; this implies that $f_n \converge_{n \to \infty} 0$ 
uniformly on compact subsets of $\D$, so that $\Vert f_{n} \circ \phi \Vert_\infty \converge_{n \to \infty} 0$. But then:
\begin{displaymath} 
\Vert M_w C_{\phi} (f_n) \Vert _2 \leq \Vert w \Vert _2 \, \Vert f_{n} \circ \phi \Vert _\infty \converge_{n \to \infty} 0 \, .
\end{displaymath} 
This shows that $M_w C_\phi$ is compact for any $w \in H^2$.
\smallskip

Now, assume  that $\Vert \phi \Vert_\infty = 1$. We are going to show that $\phi$ is decompactifiable. 

We need to find a weight $w \in H^2$ such that the finite (since $w \in H^2$) measure $\nu = \phi^{\ast} (|w^\ast|^2 \, m)$,  namely:
\begin{displaymath} 
\nu (A) = \int_{{\phi^\ast}^{- 1} (A)} |w^\ast|^2 \, dm 
\end{displaymath} 
is Carleson (ensuring that $M_{w} C_\phi \colon H^2 \to H^2$ is bounded), but not vanishing Carleson (implying that $M_{w} C_\phi \colon H^2 \to H^2$ 
is not compact).

If $C_\phi$ is not compact, it suffices to take $w = 1$.
 
We now assume that $C_\phi$ is compact. Then $m ( \{|\phi^\ast| = 1\})  = 0$. 

This fact and the hypothesis $\Vert \phi \Vert_\infty = 1$ clearly imply that $m_\phi (\Gamma_n) > 0$ for each $n$, where $\Gamma_n$ is the annulus $\{z \in \D \tq 1 - 2^{- n} \leq |z| < 1 \}$. If we set:
\begin{displaymath} 
C_l = \{z \in \D \tq 1 - 2^{- l} \leq |z|< 1 - 2^{- l - 1}\} \, , 
\end{displaymath} 
we have $\Gamma_n = \bigcup_{l \geq n} C_l$, so that $m_\phi (C_l) > 0$ for some $l \geq n$. We can therefore find an increasing sequence $(k_n)$ of integers 
such that $m_\phi (C_{k_n}) > 0$ for each $n$. Splitting in the natural way $C_{k_n}$ into $2^{k_n}$ Hastings-Luecking boxes, we can find a sequence 
$(\xi_n)$ of points of $\T$ such that, with $\tilde W_{k_n} = \tilde W (\xi_n, 2^{- k_n})$:
\begin{displaymath} 
m_\phi (\tilde W_{k_n}) > 0 \, .
\end{displaymath} 

We define our weight $w$ as an outer function $w \in H^2$ with boundary values $w^\ast$. Let
\begin{displaymath} 
u = 1 + \sum_{n = 1}^\infty \frac{2^{- k_n}}{m_\phi (\tilde W_{k_n})} \ind_{\phi^{- 1}(\tilde W_{k_n})} \, ;
\end{displaymath} 
Then $u \geq 1$, so $\log u \geq 0$, and:
\begin{displaymath} 
0 \leq \int_\T \log u \, dm \leq \int_\T (u-1) \, dm =  \sum_{n = 1}^\infty 2^{- k_n} \leq 1  < \infty \, ;
\end{displaymath} 
Hence $\log u \in L^1 (\T)$ and there is an outer function $w \in H^2$ such that $| w^\ast |^2 = u$ (see \cite[page 24]{Duren}). 

Now, if $\nu = \phi^{\ast} (|w^\ast|^2 \, m) = \phi^\ast ( u \, m)$, we have:
\begin{displaymath} 
\nu (A) = m_\phi (A) + \sum_{n = 1}^\infty \frac{2^{- k_n}}{m_\phi (\tilde W_{k_n})} \, m_\phi (A \cap \tilde W_{k_n}) \, ,
\end{displaymath}
and $\nu$ is not a vanishing Carleson measure since, with $W_{k_n} = W (\xi_n, 2^{- k_n})$:
\begin{displaymath} 
\nu (W_{k_n}) \geq  2^{- {k_n}} \frac{m_\phi (\tilde W_{k_n} \cap W_{k_n})}{m_\phi (\tilde W_{k_n})} = 2^{- k_n} \, .
\end{displaymath} 

Let now $W = W (\xi, h)$ be an arbitrary Carleson window. Without loss of generality, we can assume $h = 2^{- N}$ for some positive integer $N$, and we 
observe that if $z \in W \cap \tilde W_{k_n}$, then $1 - 2^{- N} \leq |z| \leq 1 - 2^{- k_{n} - 1}$, implying $k_n\geq N - 1$. Hence 
$W \cap \tilde W_{k_n} =\emptyset$ for $k_n < N - 1$ and:
\begin{align*}
\nu (W ) 
& = m_\phi (W) + \sum_{k_n \geq N - 1} 2^{- k_n} \frac{m_\phi (\tilde W_{k_n} \cap W)}{m_\phi (\tilde W_{k_n})} \\
& \leq m_\phi (W) + \sum_{k_n \geq N - 1} 2^{- k_n} \leq  m_\phi (W) + \sum_{l \geq N - 1} 2^{- l} \\
& = m_\phi (W) + 4 h \, .
\end{align*}
Since $C_\phi$ is bounded, $m_\phi$ is a Carleson measure and $ m_\phi (W) = {\rm O}\, (h)$; therefore $\nu (W) = {\rm O}\, (h)$ and hence 
$\nu$ is a Carleson measure. This shows that $C_\phi$ is  decompactified by $M_w$ and that completes the proof.
\end{proof}

\noindent {\bf Remark.}  For $p \geq 1$, if we set $\tilde w = w^{2 / p}$, then $w \in H^p$ and the same proof shows that the weighted composition 
operator $M_{\tilde w} C_\phi \colon H^p \to H^p$ is bounded but not compact.
\bigskip

\noindent{\bf Acknowledgement.} L. Rodr{\'\i}guez-Piazza is partially supported by the project MTM2015-63699-P (Spanish MINECO and FEDER funds). 
Parts of this paper was made when he visited the Universit\'e d'Artois in Lens and the Universit\'e de Lille, in April 2018 and January 2019. It is his pleasure to 
thank all his colleagues in these universities for their warm welcome.

This work is also partially supported by the grant ANR-17-CE40-0021 of the French National Research Agency ANR (project Front).
\goodbreak

\smallskip\goodbreak

{\footnotesize
Pascal Lef\`evre \\
Univ. Artois, Laboratoire de Math\'ematiques de Lens (LML) EA~2462, \& F\'ed\'eration CNRS Nord-Pas-de-Calais FR~2956, 
Facult\'e Jean Perrin, Rue Jean Souvraz, S.P.\kern 1mm 18 
F-62\kern 1mm 300 LENS, FRANCE \\
pascal.lefevre@univ-artois.fr
\smallskip

Daniel Li \\ 
Univ. Artois, Laboratoire de Math\'ematiques de Lens (LML) EA~2462, \& F\'ed\'eration CNRS Nord-Pas-de-Calais FR~2956, 
Facult\'e Jean Perrin, Rue Jean Souvraz, S.P.\kern 1mm 18 
F-62\kern 1mm 300 LENS, FRANCE \\
daniel.li@univ-artois.fr
\smallskip

Herv\'e Queff\'elec \\
Univ. Lille Nord de France, USTL,  
Laboratoire Paul Painlev\'e U.M.R. CNRS 8524 \& F\'ed\'eration CNRS Nord-Pas-de-Calais FR~2956 
F-59\kern 1mm 655 VILLENEUVE D'ASCQ Cedex, FRANCE \\
Herve.Queffelec@univ-lille.fr
\smallskip
 
Luis Rodr{\'\i}guez-Piazza \\
Universidad de Sevilla, Facultad de Matem\'aticas, Departamento de An\'alisis Matem\'atico \& IMUS,  
Calle Tarfia s/n  
41\kern 1mm 012 SEVILLA, SPAIN \\
piazza@us.es
}


\begin{thebibliography} {99}

\bibitem {Amar} \'E.~Amar, A.~Lederer, 
Points expos\'es de la boule unité de $H^\infty (D)$, 
C.~R.~Acad. Sci. Paris S\'er. A--B 272 (1971), A~1449--A~1452. 

\bibitem {COMA} C.~Cowen, B.~MacCluer, 
Composition Operators on Spaces of Analytic Functions, 
Studies in Advanced Mathematics,  CRC Press (1994). 

\bibitem {DJT} J.~Diestel,  H.~Jarchow, A.~Tonge, 
Absolutely summing operators, 
Cambridge Studies in Advanced Mathematics 43, Cambridge University Press, Cambridge (1995). 

\bibitem {Duren} P.~L.~Duren, 
Theory of $H^p$ spaces, 
Dover Publ. Inc., Mineola-New York (2000).

\bibitem {GKP} E.~A.~Gallardo-Guti\'errez, R.~Kumar, J.~R.~Partington, 
Boundedness, compactness and Schatten-class membership of weighted composition operators, 
Integral Equations Operator Theory 67, no. 4 (2010), 467--479. 

\bibitem {Hu-Khoi-Zhu} B.~Hu, L.~H.~Khoi, K.~Zhu, 
Frames and operators in Schatten classes, 
Houston J. Math. 41, no. 4 (2015), 1191--1219. 

\bibitem {GDHL} G.~Lechner, D.~Li, H.~Queff\'elec, L.~Rodr{\'\i}guez-Piazza, 
Approximation numbers of weighted composition operators, 
J.~Funct.~Anal. 274, no. 7 (2018), 1928--1958.

\bibitem {LLQR-Mem} P.~Lef\`evre, D.~Li, H.~Queff\'elec, L.~Rodr{\'\i}guez-Piazza,  
Composition operators on Hardy-Orlicz spaces, 
Mem. Amer. Math. Soc. 207 (2010), no. 974.

\bibitem {LLQR-2008} P.~Lef\`evre, D.~Li, H.~Queff\'elec, L.~Rodr{\'\i}guez-Piazza,  
Some examples of compact composition operators on $H^2$, 
J. Funct. Anal. 255, no. 11 (2008), 3098--3124. 

\bibitem {LELIQURO} P.~Lef\`evre, D.~Li, H.~Queff\'elec, L.~Rodr{\'\i}guez-Piazza,  
Some new properties of composition operators associated with lens maps, 
Israel J. Math. 195 (2) (2013), 801--824.

\bibitem {Lef-Rod} P.~Lef\`evre, L.~Rodr{\'\i}guez-Piazza,  
Absolutely summing Carleson embeddings on Hardy spaces, 
Adv. Math. 340 (2018), 528--587.

\bibitem {DDHL} D.~Li, H.~Queff\'elec, L.~Rodr{\'\i}guez-Piazza, 
On approximation numbers of composition operators, 
J.~Approx.~Theory. 164, no. 4 (2012), 431--459.

\bibitem {DHL} D.~Li, H.~Queff\'elec, L.~Rodr{\'\i}guez-Piazza, 
Some examples of composition operators and their approximation numbers on the Hardy space of the bidisk, 
Trans. Amer. Math. Soc., {\sl to appear}.

\bibitem {Luecking} D.~H.~Luecking, 
Trace ideal criteria for Toeplitz operators, 
J. Funct. Anal. 73, no. 2 (1987), 345--368

\bibitem{SHAPI} J.~Shapiro, 
Composition operators and classical function theory, 
Universitext, Tracts in Mathematics, Springer-Verlag (1993).

\end{thebibliography}
\end{document}